\documentclass[12pt]{article} 

\PassOptionsToPackage{hyphens}{url}\usepackage{hyperref}
\usepackage{amsthm,amsfonts,amsmath,amssymb,graphicx,enumerate,authblk,cleveref,listings}

\theoremstyle{theorem}
\newtheorem{theorem}{Theorem}[section]

\newtheorem{lemma}[theorem]{Lemma}
\newtheorem{proposition}[theorem]{Proposition}

\theoremstyle{definition}

\usepackage[colorinlistoftodos]{todonotes}
\begin{document}

\title{Crystallizations of generalized lens spaces}

\author{Basudeb Datta}

\affil{\small Department of Mathematics, Indian Institute of Science, Bangalore 560\,012, India. \newline Institute for Advancing Intelligence, TCG CREST, Kolkata 700\,091, India. \newline Emails: {\it dattab@iisc.ac.in, bdatta17@gmail.com}.}
\date{March 22, 2026}

\maketitle

%\vspace{-6mm}
\vspace{-9mm}

\begin{abstract}
We present some natural crystallizations of the generalized lens spaces $L(p, q_1, \dots, q_n)$ for integers $p\geq 2$, $n\geq 1$ and integers $q_1, \dots, q_n$ relatively prime to $p$. These crystallizations are quotients of triangulations of the sphere $S^{2n+1}$. 
\end{abstract}

\noindent
\textbf{MSC 2020: }
57Q15; % triangulating manifolds
57Q05; %General Topology of complexes
05C15; %Coloring of Graphs and hypergraphs
05E45; %Combinatorial aspects of simplicial complexes
06A06. % Partial Order

\smallskip
\noindent
\textbf{Keywords:} Lens spaces; generalized lens spaces; simplicial cell complexes; CW-complexes; crystallizations; colored graphs. 
\section{Introduction}

A simplicial cell structure $K$ of a $d$-dimensional manifold $M$ contains at least $d+1$ vertices. If $K$ has exactly $d+1$ vertices and $K$ induces a pl structure on $M$ then $K$ is called a {\em crystallization} of $M$. In that case, the dual graph $\Lambda(K)$  has a natural proper edge coloring $j$ with $d+1$ colors. In the literatures, the $(d+1)$-regular colored graph $(\Lambda(K), j)$ is called a crystallization of $M$. In \cite{Pe74}, Pezzana showed existence of crystallization for every closed connected PL manifold.  For $p  \equiv \pm 1$ (mod $q$), crystallization of the lens space $L(p, q)$ was constructed in \cite{BD14}. In Section \ref{sec:construction}, we construct  crystallizations of the generalized lens spaces $L(p, q_1, \dots, q_n)$ for integers $p\geq 2$, $n\geq 1$ and integers $q_1, \dots, q_n$ relatively prime to $p$. The derived subdivisions of these crystallizations give us triangulations of all generalised lens spaces. These crystallizations are more natural than any other simplicial cell structures or triangulations of (generalised) lens spaces. 
These crystallizations are quotient of triangulations of the sphere $S^{2n+1}$. The case where $p=2$ is known. Recently, it is shown in \cite{AB26} that $L(2, 1, \dots, 1)$ is the unique facet minimal crystallization of $\mathbb{RP}^{2n+1}$. 

%\newpage

\section{Preliminaries}

By a poset, we mean  a finite set $X$ with a partial ordering $\leq$.  By $\alpha<\beta$, we mean $\alpha\leq \beta$ and $\alpha\neq \beta$. A bijection $f: (K, \leq) \to (L, \leq)$ between two posets is called an {\em isomorphism} if $a\leq b$ if and only if $f(a) \leq f(b)$ for $a, b\in K$. In such a case, we say $(K, \leq)$ and $(L, \leq)$  are {\em isomorphic}. An isomorphism $f : (K, \leq) \to (K, \leq)$ is called an {\em automorphism} of $(K, \leq)$. The set $\operatorname{Aut}(K)$ of all the automorphisms of $(K, \leq)$ is a group under composition and is called the {\em automorphism group} of $(K, \leq)$. Any subgroup of $\operatorname{Aut}(K)$ is called a {\em group of automorphisms} of $K$. 

A CW-complex is said to be {\em regular} if all of its closed cells are homeomorphic to closed balls. Given a CW-complex $M$, let ${\mathcal M}$ be the set of all closed cells of $M$ together with the empty set $\emptyset$. Then ${\mathcal M}$ is a poset, where the partial ordering is given by set inclusion. This poset ${\mathcal M}$ is said to be the {\em face poset} of $M$. If $M$ and $N$ are two finite regular CW-complexes with isomorphic face posets then $M$ and $N$ are homeomorphic. A regular CW-complex $M$ is said to be {\em simplicial}, if the {\em boundary complex} $\partial \sigma := \{\alpha\in {\mathcal M} : \alpha < \sigma\}$ of each cell $\sigma$ of $M$ is isomorphic (as a poset) to the boundary complex of a simplex.

A {\em simplicial cell complex} or {\em simplicial poset} $K= (K,\leq)$ of dimension $d$ is a poset isomorphic to the face poset ${\mathcal M}$ of a $d$-dimensional simplicial CW-complex $M$. The topological space $M$ is called the {\em geometric carrier} of $K$ and is also denoted by $|K|$. If a topological space $N$ is homeomorphic to $|K|$, then $K$ is said to be a {\em simplicial cell structure} (or {\em simplicial cell decomposition} or {\em pseudotriangulation}) {\em of $N$}.  Simplicial posets are also referred to in the literature as {\em CW posets}, {\em pseudo-simplicial complexes} or {\em pseudocomplexes} (see \cite{Bj84, DS26, Mu13} for details). For $\alpha, \beta\in K$, if $\alpha\leq\beta$, then we say that $\alpha$ is a {\em face} of $\beta$. For a $d$-dimensional simplicial poset $K$, let $f_j(K)$ denote the number of $j$-cells of $K$ for $0\leq j\leq d$. The $(d+1)$-tuple $f(K) := (f_0(K), \dots, f_d(K))$ is called the {\it face vector} or {\em $f$-vector} of $K$. 

For a cell $\alpha$ in a simplicial poset $K$, $\partial \alpha := \{\beta \in K \, : \, \beta <\alpha\}$ is a simplicial poset and is said to be the {\em boundary} of $\alpha$. If $\partial \alpha$ is isomorphic to the boundary complex of a $j$-simplex, then $\alpha$ is a called a {\em $j$-cell} of $K$. The 0-cells and 1-cells of a simplicial poset $K$ are called the {\it vertices} and {\em edges} of $K$ respectively. The set $V(K)$ of vertices of a simplicial poset $K$ is called the {\em vertex set} of $K$. Two vertices of $K$ are called {\em adjacent} in $K$ if they are faces of an edge. For two cells $\alpha\neq \beta$, the intersection $\partial\alpha \cap\partial\beta$  may not be a single cell  (and its faces). If $\partial\alpha \cap\partial\beta$ is a cell  (and its faces) for any two cells $\alpha\neq\beta$ then the simplicial post is a simplicial complex. Cells of a simplicial complex are called {\em simplices}. For a  simplex $\alpha$ in a simplicial complex $X$, $V(\alpha) := \{v\in V(X) \, : \, v\leq \alpha\}$ is called the {\em vertex set} of $\alpha$. Since a simplex $\alpha$ of a simplicial complex $X$  is uniquely determined by $V(\alpha)$, we identify $\alpha$ with $V(\alpha)$.

A $d$-dimensional simplicial cell complex is called {\em pure} if all its maximal cells are $d$-cells. In a $d$-dimensional simplicial cell complex $K$, the $d$-cells and the $(d-1)$-cells are called {\em facets} and {\em ridges} respectively. 
The {\it dual graph} $\Lambda(K)$ of a pure simplicial cell complex $K$ is the graph whose vertices are the facets of $K$ and edges are the ordered pairs $(\{\sigma_1, \sigma_2\}, \alpha)$, where $\alpha$ is a ridge and is a common face of the facets $\sigma_1$, $\sigma_2$.  Since each 1-cell of $K$ corresponds to a geometric $1$-simplex, it follows that the graph $\Lambda(K)$ has no loops. Observe, however, that $\Lambda(K)$ can have multiple edges running between the same two vertices. 

Clearly, a simplicial cell structure $K$ of a (closed) $d$-manifold $M$ contains at least $d+1$ vertices. If $K$ has exactly $d+1$ vertices and $K$ induces a pl structure on $M$ then $K$ is called a {\em crystallization} of $M$. Suppose $X$ is a crystallization of a PL $d$-manifold $N$. Let the vertex set $V(X)$ of $X$ be $\{0, 1, \dots, d\}$. For any facet $\sigma$ of $X$ and $0\leq i\leq d$, let $\sigma_i$ be  the $(d-1)$-face of $\sigma$ opposite to the vertex $i$. 
Clearly, if $\alpha$ is a ridge of $X$ then $\alpha= \beta_i= \gamma_i$ for two facets $\alpha, \beta$ and some vertex $i$. Then, $j((\{\beta, \gamma\}, \alpha))=i$ defines a proper edge coloring $j$ on $\Lambda(X)$ with $d+1$ colors. In the literatures, the $(d+1)$-regular colored graph $(\Lambda(X), j)$ is called a crystallization of $N$ (see \cite{FGG86} for more). In \cite{Pe74}, Pezzana showed that  every closed connected PL manifold has a crystallization.

Suppose $(X, \leq)$ is the face poset of the simplicial CW-complex $M$. Let $M^{\prime}$ be the derived subdivision of $M$. The face poset of $M^{\prime}$ is called the {\it derived subdivision} of $(X,\leq)$ and is denoted by $(X^{\prime}, \leq)$. Since $M^{\prime}$ is a geometric simplicial complex, it follows that $(X^{\prime}, \leq)$ is a simplicial complex. 
(Clearly, the vertices of $X'$ have a one to correspond to the non-empty cells of $X$.)
Since the topological space $M$ and $M^{\prime}$ are same, we can identify $|X^{\prime}|$ and $|X|$ for a simplicial poset $X$.

A simplicial complex $X$ is called a {\em combinatorial manifold} if the geometric carrier $|X|$ with induced pl structure is a PL manifold. In that case, we say $X$ is {\em combinatorial} (or {\em pl}) {\em triangulation} of $|X|$. Therefore, if $|Y|$ is a PL manifold for a simplicial cell complex $Y$, then $Y^{\prime}$ is a combinatorial manifold. 

Let $X$ and $Y$ be two simplicial complexes with disjoint vertex sets. Since $\emptyset\in X, Y$,  $X\ast Y := \{\alpha\cup\beta\, :\, \alpha\in X, \beta\in Y\}$ is a simplicial complex (called the {\em join} of $X$ and $Y$) containing both $X$ and $Y$. Note that the vertex set $V(X\ast Y) = V(X) \sqcup V(Y)$ and $\dim(X\ast Y) = \dim(X) + \dim(Y) +1$. If both $X$ and $Y$ are (pl) triangulations of spheres then $X\ast Y$ is also a (pl) triangulation of a sphere (cf. \cite{RS82}). 

Let $G \leq \operatorname{Aut}(X)$ be a group of automorphisms of a poset $X$. Suppose, $X$ is the face poset of a CW-complex $M$. Then $G$ has a natural action on $M$ and $M/G$ is again a CW-complex. Clearly, the quotient poset $X/G$ is isomorphic to the face poset of the CW-complex $M/G$. In general, $X/G$ need not be a simplicial cell complex for a simplicial cell complex $X$ and $G\leq \mbox{Aut}(X)$. In \cite{DS26}, the authors have defined good action and shown that the good actions on the category of simplicial cell complexes are natural actions. The action of $G\leq \mbox{Aut}(X)$ on simplicial cell complex $X$ is {\em good}, if two vertices of $X$ in the same $G$-orbit are not adjacent in $X$. And it was  shown  the following in \cite{DS26}. 

\begin{proposition} \label{prop:quotient}
Let $X$ be a simplicial cell complex, and let $G \leq \operatorname{Aut}(X)$. 
The quotient poset $X/G$ is a simplicial cell complex if and only if the action of $G$ on $X$ is good. 
\end{proposition}

\begin{proposition} \label{prop:good}
 Let $X$ be a simplicial cell complex, and let $G \leq \operatorname{Aut}(X)$. 
 If the action of $G$ on $X$ is good, then $\left | X/G \right | \cong |X|/G$. 
\end{proposition}

For $n\geq 1$, let $S^{\hspace{.2mm}2n+1} = \{(z_0, z_1, \dots, z_n)\in \mathbb{C}^{\hspace{.2mm}n+1} \, : \, |z_0|^2 + |z_1|^2+ \cdots + |z_n|^2 =1\}$ be the unit $(2n+1)$-sphere and let $q_1, \dots, q_n$ be integers relatively prime to an integer $p\geq 2$. Define $h: S^{\hspace{.2mm}2n+1}  \to S^{\hspace{.2mm}2n+1} $ by 
\begin{align}\label{eq:action}
h(z_0, z_1, \dots, z_n) & := (e^{2\pi i/p}z_0, e^{2\pi iq_1/p}z_1, \dots, e^{2\pi iq_n/p}z_n). 
\end{align}
Then $h$ is a homeomorphism of $S^{\hspace{.2mm}2n+1}$ of period $p$. Therefore, we get  a $\mathbb{Z}_p=\langle h\rangle$ action on $S^{\hspace{.2mm}2n+1}$. This action has no fixed points. The quotient space $L(p, q_1, \dots, q_n) := S^{\hspace{.2mm}2n+1}/ \langle h\rangle$ is a $(2n+1)$-dimensional smooth (and hence PL) manifold and is called a {\em generalized lens space}  (see \cite{Sp90}). For $p, q$ relitively prime, the 3-manifold $L(p,q)$ is called a {\em lens space}.

\section{Constructions of Crystallizations}\label{sec:construction}

Throughout this section, $n\geq 1$ and $p\geq 2$ are integers. Empty set is a cell of dimension $-1$ belongs to all simplicial cell complexes. We also identify, a simplicial complex with the corresponding geometric simplicial complex. 

Let $W =\{ w_{\ell} = e^{\ell\pi i/p} \, : \, 0\leq \ell\leq 2p-1\} \subseteq S^{\hspace{.2mm}1}\subseteq \mathbb{C}$ be the set of $2p$-th roots of unity. Let $U= \{w_{2k} \, :\, 0\leq k\leq p-1\}$ and $V= W\setminus U$. So, $U = \{e^{2k\pi i/p} \, : \, 0\leq k\leq p-1\}$ is the set of $p$-th roots of unity.  

For $0\leq j\leq n$, let $\mathbb{C}_j := \{(z_0, \dots, z_{j-1}, z_j, z_{j+1}, \dots, z_{n})\in \mathbb{C}^{n+1} \, : \, z_{k} = 0$ for $k\neq j\}$. Then $\mathbb{C}_j$ is an 1-dimensional subspace of $\mathbb{C}^{n+1}$ and $\mathbb{C}^{n+1} = \oplus_{j=0}^n\mathbb{C}_j$. Let $S_j^{\hspace{.2mm}1}$ be the unit circle in  $\mathbb{C}_j$. So, $S_j^{\hspace{.2mm}1} = \{(z_0, \dots, z_{j-1}, z_j, z_{j+1}, \dots$, $z_{n+1}) \, : \, |z_j|=1, z_k= 0$ for $k\neq j\}$. Let  
\begin{align*}
W_j &:= \{w_{j, \ell} := (z_0, \dots, z_{j}, \dots, z_{n})\in S_j^{\hspace{.2mm}1},   z_j = e^{\ell\pi i/p} \, : \, 0\leq \ell \leq 2p-1\}, \\
U_j &:= \{w_{j,2k} \, : \, 0\leq k\leq p-1\}, \mbox{ and } V_j  := W_j\setminus U_j. 
\end{align*}
Now, consider the simplicial complexes 
\begin{align}\label{eq:Sigma}
\Sigma_j & := W_j \cup \{\emptyset,  \{w_{j,2p-1}, w_{j,0}\}, \{w_{j,\ell}, w_{j,\ell+1}\} : 0\leq \ell \leq 2p-2\}, \mbox{ for } 0\leq j\leq n,  \mbox{ and} \nonumber \\
\Sigma & := \Sigma_0 \ast \Sigma_1\ast \cdots \ast \Sigma_n. 
\end{align}

Then $\Sigma_j$ is a $2p$-vertex triangulation of $S_j^{\hspace{.2mm}1}$, for $0\leq j\leq n$, and hence $\Sigma$ is a $2p(n+1)$-vertex triangulation of $S^{\hspace{.2mm}2n+1}$.  
In fact,  $\Sigma$ is the boundary complex of the simplicial $2(n+1)$-polytope $P$ whose vertex set is $W_0\cup\cdots\cup W_n$. If $L_{j,\ell} := w_{j,\ell}w_{j, \ell+1}$ is the straight line segment between the vertices $w_{j,\ell}$ and $w_{j, \ell+1}$ (additions in the suffices are modulo $2p$), then the $(2n+1)$-dimensional faces of $P$ are $L_{0,\ell_0}L_{1,\ell_1}\cdots L_{n,\ell_{n}}$ (join of $n+1$ geometric 1-simplices), $\ell_0, \dots, \ell_n\in \{0, \dots, 2p-1\}$. So, the boundary  $\partial P$ of $P$ is the union of these $(2p)^{n+1}$ geometric $(2p-1)$-simplices, and we take $|\Sigma| = \partial P$. Then 
\begin{align}\label{eq:Phi}
\Phi & :  |\Sigma| \to S^{\hspace{.2mm}2n+1}, \mbox{ given by } 
  w \mapsto w/\|w\|, 
 \end{align}
 is a homeomorphism. 

Consider the map $\rho: \cup_{j=0}^nW_j \to\cup_{j=0}^nW_j$ as 
\begin{align*}
\rho(v) := \left\{\begin{array}{lcl} 
e^{2\pi i/p}\times v & \mbox{ if } & v\in W_0,  \mbox{ and} \\
 e^{2\pi iq_j/p}\times v & \mbox{ if } & v\in W_j \mbox{ for } 1\leq j\leq n. 
 \end{array}
 \right.
 \end{align*}
For $\alpha = \{v_1, \dots, v_m\}\in \Sigma$, we define 
\begin{align}\label{action}
\rho(\alpha) := \{\rho(v_1), \dots, \rho(v_m)\}. 
\end{align}

\begin{lemma}\label{lem:action}
The map $\rho$ is an automorphism of the simplicial poset $\Sigma$ and $\langle\rho\rangle \cong \mathbb{Z}_p$.
\end{lemma}

\begin{proof}
To show that $\rho\in {\rm Aut}(\Sigma)$, it is sufficient to show that $\alpha\in \Sigma$ implies $\rho(\alpha)\in \Sigma$. From the definition of join, we can write  $\alpha = \alpha_1\sqcup\cdots \sqcup \alpha_r$, where $\alpha_s\in \Sigma_{i_s}$, $s=1, \dots, r$ and $i_1, \dots, i_r$ are distinct. Since each $\Sigma_j$ is 1-dimensional, it follows that each $\alpha_r$ is a vertex or an edge. Now for any $j$, any edge $\{w_{j, k}, w_{j, k+1}\}$ of $\Sigma_j$, and any integer $t$, $\{e^{2\pi it/p}w_{j, k}, e^{2\pi it/p}w_{j, k+1}\}$ is again an edge of $\Sigma_j$. Therefore, $\rho(\alpha_{s})\in \Sigma_{i_s}$. These imply, $\rho(\alpha) = \rho(\alpha_1)\sqcup\cdots\sqcup\rho(\alpha_r) \in \Sigma$. Thus $\rho\in {\rm Aut}(\Sigma)$. Clearly, the order of $\rho$ is $p$. This proves the lemma. 
\end{proof}

Let $w\in |\Sigma|$. Suppose $w$ is in the relative interior of a simplex $\alpha= \{w_1, \dots, w_m\}$ of $\Sigma$. Let $w = t_1w_1+\cdots + t_mw_m$, for some $t_1, \dots, t_m\in (0, 1)$ with $t_1+\cdots+t_m=1$. Then $t_1\rho(w_1)+ \cdots + t_m\rho(w_m)$ is in the relative interior of $\rho(\alpha)$. Then 
\begin{align}\label{eq:|rho|}
|\rho|(w) := t_1\rho(w_1)+ \cdots + t_m\rho(w_m)  
\end{align}
gives a (pl) homeomorphism $|\rho| : |\Sigma| \to |\Sigma|$.

\begin{lemma}\label{lem:equality}
Let $h: S^{\hspace{.2mm}2n+1}  \to S^{\hspace{.2mm}2n+1} $,  $\Phi: |\Sigma| \to S^{\hspace{.2mm}2n+1}$ and $|\rho|: |\Sigma| \to |\Sigma|$ be as in \eqref{eq:action},  \eqref{eq:Phi} and \eqref{eq:|rho|} respectively. Then $\Phi\circ|\rho|\circ\Phi^{-1}\equiv h$.
\end{lemma}

\begin{proof}
Let $w\in |\Sigma|$. Suppose, $w$ is in the relative interior of $\alpha$. Let $\alpha = \alpha_1 \sqcup\cdots \sqcup \alpha_r$, where $\alpha_s$ is either a vertex or an edge of $\Sigma_{i_s}$, $s=1, \dots, r$ and $i_1, \dots, i_r$ are distinct. Then $w= a_1x_1+\cdots + a_rx_r$, where $x_1, \dots, x_r$ are in the relative interior of $\alpha_1, \dots, \alpha_r$ respectively and $a_1, \dots, a_r\in (0, 1)$ with $a_1+\cdots+a_r=1$.  

\medskip 

\noindent {\bf Claim 1.} $|\rho|(w) = a_1e^{2\pi iq_{i_1}/p} x_1+ \cdots + a_re^{2\pi iq_{i_r}/p}x_r$, where $q_0=1$. 

\medskip

Without loss, let us assume that $\alpha_1 = \{w_1, w_2\}, \dots, \alpha_{\ell} = \{w_{2\ell - 1}, w_{2\ell}\}$, $\alpha_{\ell+1}= w_{2\ell+1}, \dots, \alpha_{s} = w_{\ell+s}$. Let $x_k = b_{2k-1}w_{2k-1}+ b_{2k}w_{2k}$, where $b_{2k-1}, b_{2k}>0$ and $b_{2k-1}+ b_{2k} =1$ for $1\leq k\leq\ell$. Then $w= a_1b_1w_1 + a_1b_2w_2 + \cdots + a_{\ell}b_{2\ell}w_{2\ell}+ a_{\ell+1}w_{2\ell+1} + \cdots+ a_sw_{\ell+s}$ and $a_1b_1 + \cdots + a_{\ell}b_{2\ell}+ a_{\ell+1} + \cdots+ a_r = a_1(b_1+b_2) + \cdots + a_{\ell}(b_{2\ell-1}+b_{2\ell}) + a_{\ell+1} + \cdots+ a_r = a_1+a_2+\cdots +a_r =1$.  Then, by the definition in \eqref{eq:|rho|}, 
\begin{align*}
|\rho| (w)  &= a_1b_1\rho(w_1) + a_1b_2\rho(w_2) + \cdots + a_{\ell}b_{2\ell}\rho(w_{2\ell}) + a_{\ell+1}\rho(w_{2\ell+1}) + \cdots+ a_r\rho(w_{\ell+r}) \\
& = a_1e^{2\pi iq_{i_1}/p}(b_1w_1 + b_2w_2) + \cdots +  a_{\ell}e^{2\pi iq_{i_{\ell}}/p}(b_{2\ell-1}w_{2\ell-1} +b_{2\ell}w_{2\ell}) \\
& ~~~~ \qquad + a_{\ell+1}e^{2\pi iq_{i_{\ell+1}}/p}w_{2\ell+1}  + \cdots+ a_re^{2\pi iq_{i_{r}}/p}w_{\ell+r} \\
& = a_1e^{2\pi iq_{i_1}/p} x_1+ \cdots + a_re^{2\pi iq_{i_r}/p}x_r. 
\end{align*}
This proves Claim 1. 

\smallskip

Now, $x_s\in \mathbb{C}_{i_s}$, for $s=1,\dots, r$ and $i_1, \dots, i_r$ are distinct. Therefore, $\{a_1x_1, \dots, a_sx_r\}$ and $\{a_1e^{2\pi iq_{i_1}/p} x_1, \dots, a_se^{2\pi iq_{i_r}/p}x_r\}$ are two sets of orthogonal vectors in $\mathbb{C}^{n + 1}$. This implies that $\||\rho|(w)\|^2 = \|a_1e^{2\pi iq_{i_1}/p} x_1\|^2 + \cdots + \|a_re^{2\pi iq_{i_r}/p}x_r\|^2 = \|a_1x_1\|^2 + \cdots + \|a_rx_r\|^2=\||w\|^2$. For $1\leq s\leq r$, let $a_sx_s/\|w\|= c_se_{i_s}$, where $e_0=(1, 0, \dots, 0), e_1=(0, 1, 0, \dots, 0)$, $e_n=(0, \dots, 0, 1)$ are the standard unit vectors in $\mathbb{C}^{n+1}$. Then 
\begin{align*}
(\Phi\circ|\rho|)(w) & = \frac{a_1|\rho|(x_1) + \cdots + a_r|\rho|(x_r)}{\||q|(w)\|} = \frac{a_1e^{2\pi iq_{i_1}/p} x_1+ \cdots + a_re^{2\pi iq_{i_r}/p}x_r}{\|w\|}  \mbox{ and} \\
(h\circ\Phi)(w)& = h(c_1e_{i_1}+\cdots+c_se_{i_s}) = e^{2\pi i q_{i_1}/p}c_1e_{i_1}+\cdots+ e^{2\pi i q_{i_r}/p}c_re_{i_r}. 
\end{align*}
Therefore, $(\Phi\circ|\rho|)(w)= (h\circ\Phi)(w)$ for all $w\in |\Phi|$. Hence $\Phi\circ|\rho| \cong h\circ\Phi$. This proves the lemma. 
\end{proof}

\begin{theorem}
For integers $n\geq 1$ and $p\geq 2$, let $q_1, \dots, q_n$ be integers relatively prime to  $p$. Let $\Sigma$ and $\rho$ be as in \eqref{eq:Sigma} and \eqref{action} respectively. Then  the quotient poset $\Sigma/\langle\rho\rangle$ is a crystallization of the generalised lens space $L(p, q_1, \dots, q_n)$. 
\end{theorem}

\begin{proof}
Here  $\Sigma$ a simplicial cell complex and, by Lemma \ref{lem:action},  $\langle\rho \rangle \leq {\rm Aut}(\Sigma)$. 

\medskip 

\noindent {\bf Claim 2.} The action of $\langle\rho \rangle$ on the simplicial cell complex $\Sigma$ is good. 

\medskip

Observe that the vertices of $\Sigma$ form $2(n+1)$ $\langle\rho \rangle$-orbits, namely, $U_0, V_0$, $U_1$, $V_1$, $\dots, U_n, V_n$.  Now, no two vertices of $U_j$ (respectively, $V_j$) are adjacent in $\Sigma_j$ for $0\leq j\leq n$. Thus, from the definition of join, no two elements of any $\langle\rho \rangle$-orbit are adjacent in $\Sigma$. Therefore,  the action of $\langle\rho \rangle$ on $\Sigma$ is good. This proves Claim 2. 
 
 \medskip
 
By Proposition \ref{prop:quotient} and Claim 2, $\Sigma/\langle\rho \rangle$ is a simplicial cell complex. Then, by Proposition \ref{prop:good} , $|\Sigma/\langle\rho\rangle| \cong |\Sigma|/\langle|\rho|\rangle \cong S^{\hspace{.2mm}2n+1}/\langle \Phi\circ|\rho|\circ\Phi^{-1}\rangle$.  Now, by Lemma \ref{lem:equality}, $S^{\hspace{.2mm}2n+1}/\langle \Phi\circ|\rho|\circ\Phi^{-1}\rangle = S^{\hspace{.2mm}2n+1}/\langle h\rangle = L(p, q_1, \dots, q_n)$. Therefore, $|\Sigma/\langle\rho\rangle| \cong L(p, q_1, \dots, q_n)$. Thus, $\Sigma/\langle\rho \rangle$ is a simplicial cell structure of  $L(p, q_1, \dots, q_n)$. 

Let $\eta : \Sigma \to \Sigma/\langle\rho \rangle$ be the natural projection. Then $\eta$ is a $p$-fold cover and images of cells are cells. Let $\Sigma^{\prime}$  and $(\Sigma/\langle\rho \rangle)^{\prime}$ be the derived subdivisions of $\Sigma$  and $\Sigma/\langle\rho \rangle$ respectively. Then $\eta$ gives a $p$-fold simplicial covering $\tilde{\eta} : \Sigma^{\prime} \to (\Sigma/\langle\rho \rangle)^{\prime}$. Since each $\Sigma_j$ is triangulation of $S^{\hspace{.2mm}1}$, each $\Sigma_j$ is a combinatorial manifold and hence $\Sigma= \Sigma_0\ast \cdots\ast \Sigma_n$ is a combinatorial manifold. This implies that $\Sigma^{\prime}$ is a combinatorial manifold. Since $\tilde{\eta}: \Sigma^{\prime} \to (\Sigma/\langle\rho \rangle)^{\prime}$ is a simplicial covering, $(\Sigma/\langle\rho \rangle)^{\prime}$ is a combinatorial manifold. Therefore,  $ |\Sigma/\langle\rho \rangle| = |(\Sigma/\langle\rho \rangle)^{\prime}|$ is a PL manifold. 

Now, the vertices  of $\Sigma/\langle\rho \rangle$ are the $\langle\rho \rangle$-orbits of vertices of $\Sigma$ and hence the vertices are $U_0, V_0, U_1, V_1, \dots, U_n, V_n$. So, the number of vertices in the $(2n+1)$-dimensional simplicial cell complex $\Sigma/\langle\rho \rangle$ is $2(n+1)$. This implies that $\Sigma/\langle\rho \rangle$ is a crystallization of $L(p, q_1, \dots, q_n)$. This complete the proof. 
\end{proof}

%\subsection*{Acknowledgements}
%\newpage
%\section*{Statements and Declarations}
%The authors state that there is no conflict of interest.

{\small

\end{document}